\documentclass[10pt]{amsart}
\usepackage{amsmath}
\usepackage{amsthm}
\usepackage{amssymb}
\usepackage{amscd}
\usepackage{amsfonts}
\usepackage{amsbsy}

\newcommand{\R}{\mathbb{R}}

\newcommand{\N}{\mathbb{N}}
\newcommand{\s}{\mathbb{S}^1}
\newcommand{\T}{\mathbb{T}^2}
\newcommand{\8}{\infty}
\newcommand{\la}{\lambda}

\newtheorem{thm}{Theorem}[section]

\newtheorem{lem}[thm]{Lemma}

\newtheorem{conj}{Conjecture}

\begin{document}

\title{Invariant measures for Cherry flows}
\author{Radu Saghin}
\address{Centre de Recerca Matematica, Apartat 50, Bellaterra, Barcelona, 08193, Spain}
\email{rsaghin@gmail.com}
\author{Edson Vargas}
\address{IME-USP Sao Paulo, Rua do Matao 1010, Cidade Universitaria, Sao Paulo, Brazil}
\email{vargas@ime.usp.br}

\begin{abstract}
We investigate the invariant probability measures for Cherry flows, i.e. flows on the two-torus which have a saddle, a source, and no other fixed points, closed orbits or homoclinic orbits. In the case when the saddle is dissipative or conservative we show that the only invariant probability measures are the Dirac measures at the two fixed points, and the Dirac measure at the saddle is the physical measure. In the other case we prove that there exists also an invariant probability measure supported on the quasi-minimal set, we discuss some situations when this other invariant measure is the physical measure, and conjecture that this is always the case. The main techniques used are the study of the integrability of the return time with respect to the invariant measure of the return map to a closed transversal to the flow, and the study of the close returns near the saddle.
\end{abstract}
\date{\today}
\maketitle

\section{Introduction and results}

An important part in understanding a dynamical system consists in understanding its invariant measures and their properties. A special role is played by the physical measures, because they describe the statistical properties of a large set of trajectories. A general dynamical systems may have infinitely many complicated invariant measures in a robust way (if it has a horseshoe for example); a system may also have infinitely many or no physical measures, however it is conjectured that generically there are finitely many physical measures and that they behave 'nice' (the Palis conjecture, see \cite{Pa} for example).

For homeomorphisms in dimension one the situation is less complicated. The homeomorphisms of the circle and the interval have the invariant measures supported on periodic orbits or a minimal set, while the physical measures are supported on periodic attractors, or on the minimal set, or they don't exist. For homeomorphisms on surfaces and non-invertible maps in dimension one the situation may be already much more complicated (there may be horseshoes, the physical measure may be supported even at a repeller). The flows on surfaces are situated somewhere in between these two situations. If there are no fixed points then the dynamics can be reduced to a homeomorphism of the circle. If there are attracting fixed points or periodic orbits, then one has physical measures supported on them. If there are fixed points and non-trivial recurrence then interesting new situations can appear, and we studied some of them in \cite{SSV} - stopping a transitive flow at a point, and performing a Hamiltonian saddle-node bifurcation. The aim of this paper is to study another type of flows on surfaces with singularities and non-trivial recurrence, namely the Cherry flows. One motivation for the investigation of these specific situations is to understand the ergodic properties of flows on surfaces in general. Every flow on a surface can be broken up into some simpler pieces, some possibly corresponding to generalized Cherry flows (more saddles, sources or sinks), or corresponding to suspensions over generalized interval exchange maps, with indifferent fixed points or saddles (of course there may also be wondering regions or invariant annuli, but these are easy to understand).

We remind that if $\phi$ is a continuous flow on a compact manifold $M$, then a Borel measure $\nu$ on $M$ is called an {\it invariant measure} if $\nu(A)=\nu(\phi_t(A))$ for every measurable set $A$ and every $t\in\R$. The {\it basin of attraction} of an invariant probability measure $\nu$ is the set of points $x\in M$ such that for every continuous function $f:M\rightarrow\R$ we have
$$
\lim_{t\rightarrow\8}\int_0^tf(\phi_s(x))ds=\frac 1t\int_Mfd\nu.
$$
An invariant probability measure is called a {\it physical measure} if its basin of attraction has positive Lebesgue measure.

Throughout the paper we will consider $\phi$ to be a $C^{\8}$ (less regularity is sufficient in many cases) Cherry flow on the two-torus $\T$ with a saddle $a$ and a source $b$ (in some papers $b$ is considered to be a sink, but this would automatically correspond to the unique physical measure, which is not interesting for us). This means that the flow has no other fixed points, and has no closed or homoclinic orbits; it has a wondering domain and a quasi-minimal set (an invariant closed set such that the only invariant closed nontrivial subset is a fixed point). We assume that the eigenvalues of $\phi$ at $a$ are $\la_s<1<\la_u$.  We also consider $\s$ to be a smooth circle transversal to the Cherry flow $\phi$, with the first return map $g$. The map $g$ is a monotone circle map (with one discontinuity), and thus it has a well defined rotation number. In the case when the hyperbolic fixed point is dissipative or conservative (the divergence of the vector field at the saddle is non-positive), we obtain that the physical measure is supported at the saddle.

\begin{thm}\label{dissipative}
(Non-positive divergence at the saddle) Under the above assumptions, if $\la_s\la_u\leq 1$ (the dissipative and conservative cases), then $\delta_a$ and $\delta_b$ are the only ergodic invariant probability measures for $\phi$. Consequently the physical measure for $\phi$ is $\delta_a$, with the basin of attraction equal to $\T\setminus\{ b\}$.
\end{thm}

If the saddle has positive divergence, then there exists another probability invariant measure supported on the quasi-minimal set.

\begin{thm}\label{non-dissipative}
(Positive divergence at the saddle) Under the above assumptions, if $\la_s\la_u>1$, then there exist exactly three ergodic invariant probability measures for $\phi$: $\delta_a$, $\delta_b$, and a third invariant probability measure $\nu$ supported on the quasi-minimal set.
\end{thm}

Under some extra conditions we can show that this other measure is in fact the physical measure of the flow.

\begin{thm}\label{non-dissipative, bounded type}
(Restricted positive divergence at the saddle, bounded type case) Under the above assumptions, $\la_s^2\la_u>1$, and the rotation number of $g$ is of bounded type, then the invariant probability measure $\nu$ supported on the quasi-minimal set is the physical measure of $\phi$, with the basin of attraction having full Lebesgue measure.
\end{thm}

We recently became aware that this result is also contained in an unpublished preprint by Ali Tahzibi. We need the condition that $\la_s^2\la_u>1$ and the fact that the rotation number is of bounded type in order to use some distortion estimates for circle maps with flat intervals from \cite{GS}. However we have examples of Cherry flows with the physical measure supported on the quasi-minimal set in the general case of positive divergence at the saddle.

\begin{thm}\label{non-dissipative, Liouville}
(An example with positive divergence at the saddle) If $\la_s\la_u>1$, there exists a $C^{\8}$ Cherry flow $\phi$, with $\la_s$ and $\la_u$ the eigenvalues of the saddle point $a$, such that the invariant measure $\nu$ supported on the quasi-minimal set is the physical measure of $\phi$, with the basin of attraction having full Lebesgue measure.
\end{thm}

In this last case, the rotation number of the return map to the transversal can be made Liouvillean (it may actually always be Liouvillean in this specific construction). These two results suggest the following conjecture.

\begin{conj}\label{conjecture}
Under the above conditions, if $\la_s\la_u>1$, then the invariant probability measure supported on the quasi-minimal set is the physical measure of $\phi$.
\end{conj}

In Section 2 we treat the case of non-positive divergence at the saddle, in Section 3 we prove Theorem \ref{non-dissipative} and Theorem \ref{non-dissipative, bounded type}, and in the last section we present the example from Theorem \ref{non-dissipative, Liouville}, and ask some further questions.

\section{The case of non-positive divergence at the saddle}

We start with some general remarks. The first return map $g$ of the Cherry flow $\phi$ to the
transversal $\s$ is $C^{\8}$ everywhere except at a point which belongs to the stable manifold of $a$, which we assume to be $0$ (we identify $\s$ with $[-1/2,1/2]|_{-1/2\sim 1/2}$), and at this point it has a
discontinuity, with the one-sided limits being $c$ and $d$. We also consider $\psi$ to be the flow obtained by reversing the direction of $\phi$, and $f$ the first return map associated to $\psi$ and $\s$. We make the convention that $f$ has a flat interval, $U=[c,d]$, containing the points from $\s$ which don't return under $\psi$, and the image of this interval is $0$. Thus $f$ is a continuous and monotone map of the circle, which is $C^{\8}$ everywhere except at $c$ and $d$. If the eigenvalues of $a$ for $\phi$ are $\la_s<1<\la_u$ then we can assume that after a change of coordinates we have $f(x)=-(c-x)^r$ on an interval $[c-\epsilon,c]$ and $f(x)=(x-d)^r$ on an interval $[d,d+\epsilon]$, where $r=-\log\la_u/\log\la_s$. We remark that the condition $\lambda_s\lambda_u\leq 1$ is equivalent to the condition $r\leq 1$. Let $\tau$ be the first return time for $\phi$ to $\s$, which will have a logarithmic singularity at $0$, meaning that $\tau(x)$ is of order $-\log|x|$ on an interval $[\epsilon,\epsilon]$. Here we say that $u(x)$ is of order $v(x)$ on some set if there exists a constant $C>0$ such that $1/C\leq u(x)/v(x)\leq C$ for every $x$ in that given set. We assume that $f$ has irrational rotation number $\rho$, and it follows that $f$ is semi-conjugated with the rigid rotation of the circle $R_{\rho}$ by a continuous map $h$, and it has a unique invariant probability measure $\mu$ supported on the unique minimal set (this is also an invariant measure for $g$), given by $\mu=h^*(Leb)$. This measure lifts to a finite invariant measure $\nu$ for $\phi$ (or $\psi$) on $\T$, not supported at the singularities, if and only if $\int_{\s}\tau d\mu<\8$ (see \cite{SSV}). Thus the finite invariant measures for $\phi$ or $\psi$ are the Dirac measures at $a$ and $b$, $\delta_a$ and $\delta_b$, and possibly $\nu$, an invariant measure supported on the quasi-minimal set, depending whether the integral $\int_{\s}\tau d\mu$ is convergent or not.

Let $x_i=f^i(0)$ and $y_i=R_{\rho}^i(0)$. Let $y_{q_n}$ be the closest returns of $0$ under
$R_{\rho}$. This means that $q_n=q_{n-2}+a_nq_{n-1}$, where $a_n$ are the terms from the continued fraction expansion for $\rho$. If we iterate the interval $[0,y_{q_n}]$ under $R_{\rho}$, the first $q_{n+1}-1$ iterates are disjoint and cover more that a half of the circle. This implies that $\mu([0,x_{q_n}])=Leb([0,y_{q_n}])$ is of the order of $\frac 1{q_{n+1}}$. In order to find if the integral $\int_{\s}\tau d\mu$ converges, we need to estimate also the size of $x_{q_n}$.

In order to prove Theorem \ref{dissipative} we use the following result from \cite{M}. We remind that $f$ is a monotone self map of the circle with a flat interval $[c,d]$, $C^{\8}$ and strictly monotone outside this interval, and of the form $f(x)=-(c-x)^r$ on $[c-\epsilon, c]$ and $f(x)=(x-d)^r$ on $[d,d+\epsilon]$.

\begin{thm}\label{expanding}
Under the above assumptions on $f$, if $r\leq 1$, then there exists $N,k\in\N$, $\alpha>1$ such that if $x\in\s\setminus\cup_{i=0}^kf^{-i}U$ then $Df^N(x)>\alpha$.
\end{thm}

\begin{proof}[Proof of Theorem \ref{dissipative}]
The idea of the proof is that the dynamical intervals formed by the closest returns of $0$ to itself are uniformly expanded for a long time, and thus very small comparing the their size with respect to the invariant measure, which is given by the Lebesgue size of the corresponding dynamical intervals for the rigid rotation, and this fact makes the integral of the logarithm with respect to the invariant measure divergent.

We know that the conclusion of the Theorem \ref{expanding} holds. The first $q_{n+1}-1$ iterates of $[0,x_{q_n}]$ are disjoint from $0$, so $f^{q_{n+1}-1}$ is differentiable on $[0,x_{q_n}]$, and by the Mean Value Theorem there exists $z\in[0,x_{q_n}]$ such that
$$
|f^{q_{n+1}-1}([0,x_{q_n}])|=|f^{q_{n+1}-1}(x_{q_n})-f^{q_{n+1}-1}(0)|=|x_{q_n}|Df^{q_{n+1}-1}(z).
$$
Here we denote by $|I|$ the size of the interval $I$. Let $q_{n+1}-1=NA+B$ for $A,B\geq 0$ and $k<B<k+N$ (we assume that $n$ is large enough). Then
$$
Df^{q_{n+1}-1}(z)=Df^{NA}(z)Df(f^{NA}(z))Df(f^{NA+1}(z))\dots Df(f^{q_{n+1}-2}(z)).
$$
Because $r\leq 1$, there exists $\beta>0$ such that $Df(x)\geq\beta$ for $x\in\s\setminus U$. Because  the first $q_{n+1}-1$ iterates of $[0,x_{q_n}]$ are disjoint from $U$, and the first $q_{n+1}-k-1$ iterates of $[0,x_{q_n}]$ are disjoint from $\cup_{i=0}^kf^{-i}U$, Applying Theorem \ref{expanding} we obtain that
$$
Df^{q_{n+1}-1}(z)\geq\alpha^A\beta^B\geq C\alpha^{\frac{q_{n+1}}N}
$$
for some constant $C$ depending only on $\alpha,\beta,k,N$. Consequently
$$
1\geq |f^{q_{n+1}-1}([0,x_{q_n}])|=|x_{q_n}|Df^{q_{n+1}-1}(z)\geq C\alpha^{\frac{q_{n+1}}N}|x_{q_n}|,
$$
or $|x_{q_n}|\leq\frac 1C\alpha^{-\frac{q_{n+1}}N}$. Then
$$
\int_{[0,x_{q_n}]}\tau d\mu\geq C'\int_{[0,x_{q_n}]}-\log|x|d\mu\geq -C'\mu([0,x_{q_n}])\log|x_{q_n}|\geq C"\log\alpha,
$$
where $C',C"$ are constants independent of $n$, so $\int_{[0,x_{q_n}]}\tau d\mu$ does not converge to zero as $n$ tends to infinity (or $x_{q_n}$ tends to zero), or the integral $\int_{\s}\tau d\mu$ is divergent. This implies that the only ergodic invariant probability measures for $\phi$ and $\psi$ are the Dirac measure at $a$ and $b$, and because $b$ is a repeller we obtain that the physical measure for $\phi$ is $\delta_a$, with the basin of attraction $\T\setminus\{ b\}$.
\end{proof}

\section{The non-dissipative case}

We start with the proof of Theorem \ref{non-dissipative}. This proof was communicated to us by Jiagang Yang.

\begin{proof}[Proof of Theorem \ref{non-dissipative}]
We will show that if there is an open set $U$ such that the flow $\phi$ has positive divergence on $U$, then there can be no physical measure for $\phi$ supported inside $U$. As a consequence $\delta_a$ cannot be the physical measure for $\phi$ (there is positive divergence in a neighborhood), so the third invariant probability measure supported on the quasi-minimal set must exist.

Assume that there exists a physical measure $\alpha$ in $U$. If $x\in\mathcal B(\alpha)$ is in the basin of $\alpha$ we get $\lim_{t\rightarrow\infty}Jac_x\phi_t(x)=\infty$ (there is a lower bound for the divergence, and the orbit of $x$ spends most of the time in $U$). For any $C>0$ let $B(C,t)=\{ x\in\mathcal B(\alpha):\ Jac_x\phi_s(x)>C,\forall s\geq t\}$. Then $Leb(\phi_t(B(C,t)))>CLeb(B(C,t))$ or $Leb(B(C,t))<1/C$. But $B(C,t)$ is increasing with respect to $t$ and covers $\mathcal B(\alpha)$, so $Leb(\mathcal B(\alpha))\leq 1/C$, and this is for every $C>0$, which is a contradiction because $\alpha$ is a physical measure.
\end{proof}

In order to prove Theorem \ref{non-dissipative, bounded type}, we use the following results from \cite{GS}. We remind that $f$ is a monotone self-map of the circle with a flat interval $[c,d]$, $C^{\8}$ and strictly monotone outside this interval, and of the form $f(x)=-(c-x)^r$ on $[c-\epsilon, c]$ and $f(x)=(x-d)^r$ on $[d,d+\epsilon]$.

\begin{thm}\label{circle}
Under the above conditions for $f$, if $r>2$ and the rotation number $\rho$ of $f$ is of bounded type, then
$$
\liminf_{n\rightarrow\infty}\frac{|x_{q_n}|}{|x_{q_{n-2}}|}>0.
$$
\end{thm}

\begin{thm}\label{Lebesgue}
Under the above conditions for $f$, if $r\geq 1$, then the wondering set of $\phi$ (which is the basin of attraction of $b$ under $\psi$, without $b$ of course) has full Lebesgue measure on $\T$.
\end{thm}

\begin{proof}[Proof of Theorem \ref{non-dissipative, bounded type}]
The condition $\la_s^2\la_u>1$ is equivalent to $r>2$, so we are in the hypothesis of Theorem \ref{circle}, and we can assume that $\frac{|x_{q_n}|}{|x_{q_{n-2}}|}>\alpha^2$ for $n\geq n_0>0$ and some constant $\alpha\in(0,1)$. From this it is easy to prove by induction that $|x_{q_n}|\geq C\alpha^n$ for some $C>0$. We know from Theorem \ref{non-dissipative} that the flow $\phi$ has an ergodic invariant probability measure $\nu$, corresponding to the invariant measure $\mu$ for $f$, supported on the quasi-minimal set. It remains to show that $\nu$ is indeed the physical measure for $\phi$, and its basin of attraction has full Lebesgue measure.

Because of Theorem \ref{Lebesgue}, it is enough to check that the points from the wondering set of $\phi$ are in the basin of attraction of $\nu$, or just the points from the interval $[c,d]$, because every point from the wondering set passes through this interval. Fix some $z\in[c,d]$ and let $z_n=g^{n-1}(z)$ and $t_n=\tau(z_n)$. For any $t>0$ there exists $N\in\N$ such that $t=t_1+t_2+\dots +t_N+\tilde t$ where $0<\tilde t\leq t_{N+1}$. There exists $n\in\N$ such that $q_n\leq N<q_{n+1}$.

Let $m_t$ me the probability measure corresponding to the piece of trajectory of length $t$ of
$\phi$ starting at z. The possible limits of $m_t$ must be of the form $r\delta_a+(1-r)\mu$, for
some $r\in[0,1]$, because the only ergodic invariant probability measures are $\delta_a,\delta_b$
and $\mu$, and the forward trajectory of $z$ is bounded away from $b$ because $b$ is a repeller.
We will show that $r$ must be zero.

Because $\tau$ is uniformly bounded from bellow, we get that $t\geq CN$. We will fix $0<n_0<n$ and
estimate the time $t_{A_{n_0}}$ spent by the trajectory $\phi_s(z)$, $0\leq s\leq t$ inside the
neighborhood of $a$ given by $A_{n_0}=\{ \phi_s(w):\ w\in[x_{q_{n_0}}, x_{q_{n_0+1}}], 0\leq
s\leq\tau(w)\}$. If, by choosing a convenient $n_0$, this time spent inside $A_{n_0}$ can be made
arbitrarily small with respect to $t$, for any $t$ large enough, then we get that $r$ must be
arbitrarily small and we are done.

We first remark that $h(z_n)=y_{-n}$, where $h$ is the semi-conjugacy between $f$ and the rigid
rotation $R_{\rho}$. Then $z_{q_l}\in[x_{q_{l-1}},x_{q_{l+1}}]$ and
$x_{q_l}\in[z_{q_{l-1}},z_{q_{l+1}}]$, for any $l\in\N$. Also from now on we will use the same
notation $C$ for different constants which are independent of $t$ (or $N$ or $n$) and $n_0$.

The number of points $z_i$, $1\leq i\leq N$ inside $[x_{q_l},x_{q_{l+2}}]$ is equal to the number of points $y_{-i}$, $1\leq i\leq N$ inside $[y_{q_{l}},y_{q_{l+2}}]$. Because the rotation number $\rho$ of $f$ is of bounded type, the points $y_{-i}$, $1\leq i\leq N$ divide the circle into $N$ intervals which are of comparable size, bounded from bellow by $1/CN$ and from above by $C/N$, for some constant $C$ independent of $N$ (depends on $\rho$). This implies that the number of points $y_{-i}$, $1\leq i\leq N$ which are inside $[y_{q_{l}},y_{q_{l+2}}]$ is less than or equal to $CN|y_{q_l}|$. As we remarked before $|y_{q_l}|$ is of order $1/q_{l+1}$, so the number of points $z_i$, $1\leq i\leq N$ which are inside $[x_{q_l},x_{q_{l+2}}]$ is smaller than or equal to $CN/q_{l+1}$.

Then the time $t_{B_l}$ spent by the trajectory $\phi_s(z)$, $0\leq s\leq t$ inside $B_l=\{ \phi_s(w):\ w\in[x_{g_l}, x_{q_{l+2}}], 0\leq s\leq\tau(w)\}$ is bounded from above by $-CN\log|x_{q_{l+2}}|/q_{l+1}$. We remind that $|x_{q_l}|\geq C\alpha^l$, and one can also show by induction that $q_l\geq\beta^l/C$ for $\beta=\frac{1+\sqrt 5}2$, so we get $t_{B_l}\leq CN\frac l{\beta^l}$. Then
$$
\frac {t_{A_{n_0}}}t=\frac 1t\sum_{l=n_0}^{n-1}t_{B_l}\leq \frac{CN}t\sum_{l=n_0}^{n-1}\frac l{\beta^l}\leq C\sum_{l=n_0}^{n-1}\frac l{\beta^l}.
$$
But $\sum_{l=1}^{\8}\frac l{\beta^l}$ is convergent, so by taking $n_0$ large enough we can make $t_{A_{n_0}}/t$ as small as we want for any $n$ arbitrarily large, so $r=0$.
\end{proof}

\section{A non-dissipative example}

In this section we will construct the example from Theorem \ref{non-dissipative, Liouville}. The flow we need will be a limit of flows on the torus with a source and a saddle and a parabolic orbit of (quickly) increasing period. The sequence of flows is constructed by induction. The idea of the construction is to make the parabolic periodic orbit of the following flow spend most of the time very close to the parabolic periodic orbit of the previous flow, and using this fact to make sure that the forward orbit of the points from $[c,d]$ spend most of the time away from the saddle $a$; this would show that there must be an invariant probability measure supported on the quasi-minimal set, and because of Theorem \ref{Lebesgue}, this measure must be the physical one.

There exist a $C^{\8}$ (even analytic) flow $\phi_1$ on the two-torus which has a saddle $a$ and a source $b$, a transversal $\s$, the return map to the transversal $g_1$ with a discontinuity at $0$, its inverse $f_1$ with a flat interval $[c,d]$, and $g_1$ has exactly one parabolic non-degenerate fixed point $p_1$, no other fixed points, and $0\notin [c,d]$. The sequence of flows we will construct will be small perturbations of $\phi_1$, with the perturbation supported near the orbit of $p_1$, so that $a,b,c,d,0$ remain unchanged (we prefer to leave them unchanged in order to simplify notations, and this is why we work with $C^{\8}$ flows).

We start with some notations. Given a flow $\phi$, a map $f$ and a set $A$, we denote by
$$
t_A(\phi, x,t)=\int_0^t\chi_A(\phi_s(x))ds=Leb(\{s\in[0,t]:\ \phi_s(x)\in A\}),
$$
$$
n_A(f,x,n)=|\{i\in\N:\ 0\leq i\leq n, f^i(x)\in A\}|,
$$
i.e. the time spent inside $A$ by the trajectory of $\phi$ (or $f$) starting at $x$ and of length $t$ (or $n$); here $\chi_A$ is the indicator function of the set $A$, and $|A|$ is the cardinality of $A$. We fix a sequence $\{ c_n\}_{n\geq 1}$ in $(0,1)$ such that $\prod_{n=1}^{\8}c_n>0$.

Assume that we have constructed the $C^{\8}$ flows $\phi_k$, $1\leq k\leq n$ on the torus $\T$, with the following properties:
\begin{enumerate}
\item
If $X_k$ is the vector field corresponding to $\phi_k$, then $\| X_k-X_{k-1}\|_{C^{k}}<2^{-k}$ (this guarantees the $C^{\8}$ convergence of $X_k$ to a $C^{\8}$ vector field $X$);
\item
$\phi_k$ has a saddle $a$, a source $b$, $\s$ is transverse to the flow, with the return map $g_k$ which has a discontinuity at $0$, with the inverse $f_k$ which has a flat interval $[c,d]$; the return time to $\s$ is $\tau_k$;
\item
$\phi_k$ also has a non-degenerate parabolic orbit, of period $\tau_0^k$; let $p_k$ be the point of the parabolic orbit in $\s$ which is closest to $0$; $p_k$ is a non-degenerate parabolic periodic point for $g_k$ of period $b_k=b_{k-1}a_{k-1}$, where $a_{k-1}$ can be chosen arbitrarily large; there are no other periodic orbits or homoclinic or heteroclinic connections for $\phi$, (so $0$ does not belong to the forward orbit of $[c,d]$ under $g_k$);
\item
Time estimates: For every $1\leq i<n$ there exist open neighborhoods $B_i$ of $p_i$ such that
$dist(0,A_i)>0$ where
$$
A_i=\{ \phi_i(x):\ x\in B_i, 0\leq t\leq\tau_i(x)+\tau_i(g_i(x))+\dots+\tau_i(g_i^{b_i-1}(x))\}
$$
(the image of $B_i$ under the flow $\phi_i$ until it returns to $\s$ for $b_i$ times; it is a
neighborhood of the parabolic orbit of $\phi_i$); there exist $0<t_1<t_2<\dots <t_{n-1}$ such that
for any $z\in[c,d]$, for every $1\leq i<k<n$, and for every $t\geq t_i$, we have
$$
t_{A_i}(\phi_k,z,t)> c_ic_{i+1}\dots c_{k-1}t,
$$
$$
t_{A_i}(\phi_k,p_k,\tau_0^k)>c_ic_{i+1}\dots c_{k-1}\tau_0^k;
$$
furthermore $t_k$ depends only on $\phi_1,\phi_2,\dots,\phi_{k-1}$.
\end{enumerate}

Now we will explain how to construct $\phi_{n+1},p_{n+1},b_{n+1},B_n$ and $t_n$, such that the properties enumerated above hold also for $\phi_{n+1}$.

First consider a $C^{\8}$ one-parameter family of flows which unfold generically the parabolic
orbit of $\phi_n$: $\phi_n^s$, $-\epsilon<s<\epsilon$, $\phi_n^0=\phi_n$, and for $s>0$, $g_{n,s}$
(the return map of $\phi_n^s$ to $\s$) has no periodic orbit of period $b_n$. This family can be
constructed by small rotations of the vector field inside a small disk intersecting the parabolic
orbit, thus away from $a,b,c,d,0$. We can assume also that $g_{n,s}$ and $g_{n,s}^{b_n}$ are
monotone with respect to $s$, so the rotation number of $g_{n,s}^{b_n}$ is monotone and continuous
as a function of $s$. Because the parabolic orbit disappears, the rotation number is non-constant,
so there are parameters for which the rotation number of $g_{n,s}^{b_n}$ is equal to $1/a_n$ or
$-1/a_n$ for every integer $a_n$ large enough (we will assume that it is $1/a_n$, for $-1/a_n$ the
proof is similar). In fact the set of parameters for which the rotation number of $g_{n,s}^{b_n}$
is $a_n$ is an interval, and by taking $s$ to be an endpoint of this interval we know that the
periodic orbits of $g_{n,s}^{b_n}$ of period $a_n$ must be parabolic or indifferent. Here we use
the positive divergence at the saddle condition, in the case of non-positive divergence we may also have homoclinic connections
($0=g_{n,s}^{b_n-1}(c)$ for example); however in the positive divergence case the homoclinic
connections are repelling and thus the rotation number would be locally constant, so they cannot
correspond to the endpoints of intervals. A further arbitrarily small perturbation ensures that
there is only one parabolic orbit for $g_{n,s}^{b_n}$ which is non-degenerate (rotate the vector
field up outside of a fixed parabolic orbit for $\phi_n^s$, inside a neighborhood of the parabolic
orbit of $\phi_1$ for example).

Thus for every $a_n$ large enough we have a new flow $\phi_{a_n}$ which verifies the conditions
(1)-(3). Furthermore the larger $a_n$ is, the closer $\phi_{a_n}$ is to $\phi_n$, and the
parabolic orbit of $\phi_{a_n}$ spends more and more time inside any given neighborhood of the
parabolic orbit of $\phi_n$. Let $p^{a_n}$ be the point of the parabolic orbit of $\phi_{a_n}$
which is on $\s$ and closest to $0$, $b_na_n$ the period of $p^{a_n}$ under $g_{a_n}$
(the return map to $\s$), and the return time is $\tau_{a_n}$.

For $x\in\s$ let $\tau_n^k(x)=\sum_{i=0}^{k-1}\tau_n(g_n^i(x))$ ($\tau_n$ is the return time to $\s$ for $\phi_n$; $\tau_n^k$ is the time needed for $\phi_n$ to return to $\s$ $k$ times). We know that $\tau_0^n=\tau_n^{b_n}(p_n)$, so from (4) we have
$$
t_{A_i}(\phi_n,p_n,\tau_n^{b_n}(p_n))>c_ic_{i+1}\dots c_{n-1}\tau_n^{b_n}(p_n).
$$
By continuity, there exists a neighborhood $B_n$ of $p_n$ in $\s$ such that
$$
t_{A_i}(\phi_n,x,\tau_n^{b_n}(x))>c_ic_{i+1}\dots c_{n-1}\tau_n^{b_n}(x),\ \forall 1\leq i<n,\
\forall x\in \overline {B}_n.
$$
Define as before
$$
A_n=\{ \phi_n(x):\ x\in B_n, 0\leq t\leq\tau_n^{b_n}(x)\}.
$$

We will need to analyze where the forward orbit of $[c,d]$ spends most of its time. The following Lemma is the main tool we use, it says that the forward orbits of $[c,d]$ and $p^{a_n}$ are bounded away from $0$ independent from $a_n$, and it uses again the divergence condition on the saddle.

\begin{lem}\label{parabolic}
There exists $A,\delta>0$ such that for every $a_n>A$, for every $z\in[c,d]\cup\{ p^{a_n}\}$, and
for every $k>0$, we have $|g_{a_n}^k(z)|>\delta$.
\end{lem}

We postpone the proof of Lemma \ref{parabolic} for later. As a corollary of this, we get that
there exists $T_n>0$ such that for every $a_n>A$, for every $z\in[c,d]\cup\{ p^{a_n}\}$, and for
every $k>0$, we have $\tau_{a_n}(g_{a_n}^k(z))<T_n$. Let also $T_0$ be a lower bound for the return time for all $a_n$ large enough.

 Let $B_n'\subset B_n$ be a neighborhood of $p_n$ such that for every $a_n>A$ (eventually make $A$ larger), we have $C_n=\cup_{i=0}^{b_n-1}g_n^i(B_n')\subset\cup_{i=0}^{b_n-1}g_{a_n}^i(B_n)$. The set of points $\{ g_{a_n}^{kb_n}(p^{a_n}):\ 0\leq k\leq a_n-1\}$ divides the circle $\s$ into $a_n$ disjoint intervals, which are permuted by $g_{a_n}^{b_n}$. Because $g_{a_n}^{b_n}$ is close to $g_n^{b_n}$, and $p_n$ is the unique fixed parabolic orbit for $g_n^{b_n}$, for all $a_n$ large enough there are only finitely many intervals from the above mentioned family, which are not inside any given neighborhood of the parabolic orbit of $g_n$. Let $d_n$ be the number of intervals from the above partition which are not inside $C_n$, which is independent of $a_n$. This implies that for every $a_n>A$ (again eventually after making $A$ larger), we have
$$
n_{\s\setminus C_n}(g_{a_n},p^{a_n},b_na_n)\leq d_nb_n,
$$
$$
n_{\s\setminus C_n}(g_{a_n},x,N)\leq kd_nb_n,\ \forall x\in\s,\forall (k-1)a_nb_n<N\leq ka_nb_n.
$$
The last inequality implies that
$$
n_{\s\setminus C_n}(g_{a_n},x,N)\leq\frac {d_n}{a_n}N+d_nb_n,\ \forall x\in\s,\forall N\in\N.
$$

We fix $t_n$ large enough (its size depends on $b_n,d_n,c_n,T_n$ and $T_0$ and will be specified later). Let $a_n>A$ be large enough (again it will be specified later, and depends on $t_n$) such that $\phi_{n+1}=\phi_{a_n}$ is close enough to $\phi_n$ in order to
satisfy the following conditions:
$$
(a)\ t_{A_i}(\phi_{n+1},z,t)>c_ic_{i+1}\dots c_{n-1}t,\ \forall z\in[c,d],\ \forall 1\leq i<n,\ \forall t\in[t_i,t_n],
$$
$$
(b)\ t_{A_i}(\phi_{n+1},x,\tau_{n+1}^{b_n}(x))>c_ic_{i+1}\dots c_{n-1}\tau_{n+1}^{b_n}(x),\ \forall x\in \overline{B}_n,\ \forall 1\leq i<n.
$$
Here $\tau_{n+1}$ is the return map of $\phi_{n+1}$ to $\s$, and we also use the expected notations
$g_{n+1}=g_{a_n},p_{n+1}=p^{a_n},b_{n+1}=b_na_n$. For any $x\in\s$ and $N\in\N$, denote by
$$
B(x,N)=\{ i\in\N:\ 0\leq i\leq N-b_n, g_{n+1}^i(x)\in B_n'\}
$$
and
$$
D(x,N)=\{ i\in\N:\ 0\leq i\leq N-b_n, g_{n+1}^i(x)\notin C_n\}\cup\{ N-b_n+1,N-b_n+2,\dots,N\}.
$$
We remark that
$$
|D(x,N)|\leq n_{\s\setminus C_n}(g_{n+1},x,N)+b_n\leq\frac {d_n}{a_n}N+d_nb_n+b_n.
$$

In order to finish the construction by induction we have to show the time estimates (4) from $\phi_{n+1},A_n$ and $t_n$.
\begin{eqnarray*}
t_{A_n}(\phi_{n+1},p_{n+1},\tau_0^{n+1})&\geq&\sum_{y\in B(p_{n+1},b_{n+1})}t_{A_n}(\phi_{n+1},y,\tau_{n+1}^{b_n}(y)))\\
&=&\sum_{y\in B(p_{n+1},b_{n+1})}\tau_{n+1}^{b_n}(y)\geq\tau_0^{n+1}-\sum_{y\in D(p_{n+1},b_{n+1})}\tau_{n+1}(y)\\
& > & \tau_0^{n+1}-T_n|D(p_{n+1},b_{n+1})|\\
&>&\tau_0^{n+1}\left( 1-\frac{d_nT_n}{a_nT_0}\right)-b_nT_n(1+d_n).
\end{eqnarray*}
Here we used the bound from above on $|D(x,N)|$ and the fact that $b_{n+1}T_0\leq\tau_0^{n+1}$.

For every $1\leq i\leq n-1$ we have, using the condition ($b$),
\begin{eqnarray*}
t_{A_i}(\phi_{n+1},p_{n+1},\tau_0^{n+1}) & \geq & \sum_{y\in B(p_{n+1},b_{n+1})}t_{A_i}(\phi_{n+1},y,\tau_{n+1}^{b_n}(y)))\\
& > & c_ic_{i+1}\dots c_{n-1}\sum_{y\in B(p_{n+1},b_{n+1})}\tau_{n+1}^{b_n}(y)\\
& \geq & c_ic_{i+1}\dots c_{n-1}(\tau_0^{n+1}-\sum_{y\in D(p_{n+1},b_{n+1})}\tau_{n+1}(y))\\
& \geq & c_ic_{i+1}\dots c_{n-1}(\tau_0^{n+1}-T_n|D(p_{n+1},b_{n+1})|)\\
&>&c_ic_{i+1}\dots c_{n-1}\left[ \tau_0^{n+1}\left( 1-\frac{d_nT_n}{a_nT_0}\right)-b_nT_n(1+d_n)\right].
\end{eqnarray*}

So in order to get the time estimate required in (4) for the orbit of $p_{n+1}$ and $1\leq i\leq n$ we need
$$
c_n\tau_0^{n+1}<\tau_0^{n+1}\left( 1-\frac{d_nT_n}{a_nT_0}\right)-b_nT_n(1+d_n),
$$
which is definitely true if $a_n$ is large enough (this would also imply that $\tau_0^{n+1}\geq a_nb_nT_0$ is large).

In order to get the required time estimates for the orbit of $[c,d]$, we first remark that the condition ($a$) gives the inequalities for $1\leq i<n$ and $t_i\leq t\leq t_n$, so it remains to show the inequalities for $t>t_n$. For every $1\leq i<n$, $z\in[c,d]$ and $t>0$, $\tau_{n+1}^N(z)\leq t<\tau_{n+1}^{N+1}(z)$, applying again condition ($b$) we obtain
\begin{eqnarray*}
t_{A_i}(\phi_{n+1},z,t) & \geq & \sum_{y\in B(z,N)}t_{A_i}(\phi_{n+1},y,\tau_{n+1}^{b_n}(y)))\\
& > & c_ic_{i+1}\dots c_{n-1}\sum_{y\in B(z,N)}\tau_{n+1}^{b_n}(y)\\
& \geq & c_ic_{i+1}\dots c_{n-1}(t-\sum_{y\in D(z,N)}\tau_{n+1}(y))\\
& \geq & c_ic_{i+1}\dots c_{n-1}(t-T_n|D(z,N)|)\\
&>& c_ic_{i+1}\dots c_{n-1}\left[ t\left( 1-\frac{d_nT_n}{a_nT_0}\right)-b_nT_n(1+d_n)\right].
\end{eqnarray*}
We used again the bound on $|D(z,N)|$ from above and the fact that $N\leq t/T_0$. In a similar way we get
$$
t_{A_n}(\phi_{n+1},p_{n+1},t)>t\left( 1-\frac{d_nT_n}{a_nT_0}\right)-b_nT_n(1+d_n).
$$

So in order to get the time estimate required in (4) for the orbit of $z\in[c,d]$ we need
$$
c_nt<t\left( 1-\frac{d_nT_n}{a_nT_0}\right)-b_nT_n(1+d_n),\ \forall t>t_n.
$$
To finish the proof of the time estimates and thus the induction, we do the following: first fix $A_0=\frac{2d_nT_n}{a_nT_0(1-c_n)}$, so if $a_n>A_0$ then $1-c_n-\frac{d_nT_n}{a_nT_0}>\frac{1-c_n}2$; then let $t_n=\frac{2b_nT_n(1+d_n)}{1-c_n}$, so if $a_n>A_0$ and $t>t_n$ then $c_nt<t\left( 1-\frac{d_nT_n}{a_nT_0}\right)-b_nT_n(1+d_n)$; then choose $a_n>A,a_n>A_0$ such that $\tau_0^{n+1}\geq a_nb_nT_0>t_n$ and the conditions ($a$) and ($b$) are satisfied, and we are done.

We have constructed the sequence $\phi_n$, and let $\phi_0$ be the limit, which is $C^{\8}$ because of condition (1). By making $a_n$ grow fast enough we can make sure that the rotation number of $g_0$, the return map of $\phi_0$ to $\s$, is irrational, or even Liouville, so $\phi_0$ is indeed a Cherry flow. Given any $\epsilon>0$, there exists $n_0>0$ such that $\prod_{n=n_0}^{\8}c_n>1-\epsilon$. Let $3\delta=dist(a,A_{n_0})>0$ ($a$ is the saddle). Given any $t>t_{n_0}$ and any $z\in(c,d)$, there exists $n_z>n_0>0$ such that for any $s\in[0,t]$, $dist(\phi_{0,t}(x),\phi_{n_z,t}(x))<\delta$. Then
$$
t_{B(a,2\delta)}(\phi_0,z,t)<t-c_{n_0}c_{n_0+1}\dots c_{n_z}t<\epsilon t,
$$
so for any probability measure $\alpha$ which is a limit of probability measures $m_t(\phi_0,z)$ (corresponding to orbits of $\phi_0$ starting at $z$ and of length $t$), we have $\alpha(B(a,\delta))<\epsilon$, ($B(a,\delta)$ is the ball centered at $a$ and of radius $\delta$). In a similar way to the proof of Theorem \ref{non-dissipative, bounded type}, this shows that the probability invariant measure $\nu$ supported on the quasi-minimal set is the physical measure for $\phi_0$.

To complete the construction, we now give the proof of Lemma \ref{parabolic}.

\begin{proof}[Proof of Lemma \ref{parabolic}]
It is enough to prove that $|p^{a_n}|>\delta$, $|g_{a_n}^{a_nb_n-1}(c)|>\delta$, and $|g_{a_n}^{a_nb_n-1}(d)|>\delta$ for some $\delta>0$, because these are the points from the forward trajectories of $p^{a_n}$ and $[c,d]$ closest to $0$ (remember that the orbit of $p^{a_n}$ is the unique parabolic orbit of $g_{a_n}$, so the forward orbit of $[c,d]$ will converge to it). We will need the following Lemma.

\begin{lem}\label{intervals}
Let $F:[a_1,a_2]\rightarrow \R$, $p\in[a_1,a_2]$, such that $F(p)=p$, $F'>0$ on $[a_1,a_2]\setminus\{ p\}$, and $F''\neq 0$ on $[a_1,a_2]$ (p is a non-degenerate parabolic fixed point). Then there exists $K>0$ such that for any $f$ which is increasing, $C^2$ close enough to $F$ ($f''\neq 0$), $f(x)\neq x\ \forall x\in[a_1,a_2]$, and for any $x\in[a_1,F^2(a_1)]$, if we denote by $n(x,f)$ the number of iterates of $x$ under $f$ which stay inside $[a_1,a_2]$ ($f^{n(x,f)-1}(x)\leq a_2<f^{n(x,f)}(x)$), then we have $1/K\leq (f^{n(x,f)})'(x)\leq K$.
\end{lem}

\begin{proof}
This fact would probably follow from parabolic renormalization results; however there is a simple proof which we will present here.

We assume that $F'',f''>0$, the other case is similar. Then every iterate of $f$ is also increasing and convex. There exists $z\in[x,f(x)]$ such that
$$
(f^{n(x,f)-1})'(z)=\frac{f^{n(x,f)}(x)-f^{n(x,f)-1}(x)}{f(x)-x}.
$$
Then
\begin{eqnarray*}
(f^{n(x,f)})'(x) & = & f'(x)(f^{n(x,f)-1})'(f(x))\geq f'(x)(f^{n(x,f)-1})'(z)\\
& \geq & f'(x)\frac{f^{n(x,f)}(x)-f^{n(x,f)-1}(x)}{f(x)-x}\\
& \geq & \frac 12\inf_{y\in[a_1,F^2(a_1)]}F'(y)\frac{\inf_{y\in[F^{-2}(a_2),a_2]}(y-F^{-1}(y))}{\sup_{y\in[a_1,F^2(a_1)]}(F(y)-y)}\geq \frac 1K.
\end{eqnarray*}
Here we used the fact that $f$ is $C^1$ close enough to $F$, and the intervals where we took the supremum and infimum are bounded away from $p$, where potential problems would arise. The other inequality is proved similarly.
\end{proof}

Choose a small interval $U$ around $p_n$, such that $(g_n^{b_n})''\neq 0$ in $\cup_{i=0}^{b_n-1}g_n^i(U)$ (remember that $p_n$ is a non-degenerate parabolic point). We can assume that also $(g_{a_n}^{b_n})''\neq 0$ in $\cup_{i=0}^{b_n-1}g_n^i(U)$, because it is arbitrarily close to $g_n^{b_n}$. As we remarked before, only a finite number of iterates of $p^{a_n}$ under $g_{a_n}$ are outside $\cup_{i=0}^{b_n-1}g_n^i(U)$, and this number is independent of $a_n$. Applying Lemma \ref{intervals} and the fact that the derivative of $g_{a_n}$ is uniformly bounded from below independently of $a_n$, we get that $(g_{a_n}^{b^{a_n}-1})'(g_{a_n}(p^{a_n}))>C$ for some constant $C$ independent of $a_n$ for all $a_n$ large enough. But because $p^{a_n}$ is a parabolic orbit, we get that
$$
g_{a_n}'(p^{a_n})=\frac{(g_{a_n}^{b^{a_n}})'(p^{a_n})}{(g_{a_n}^{b^{a_n}-1})'(g_{a_n}(p^{a_n}))}\leq \frac 1C
$$
for any $a_n$ large enough. But near $0$ the derivative of $g_{a_n}$ tends uniformly to infinity ($g_{a_n}'(t)=\frac 1r|t|^{1-\frac 1r}$), so we get that $|p^{a_n}|\geq\delta_0$ for some $\delta_0>0$ independent of $a_n$.
 
Remember the partition of $\s$ made by the points $\{ g_{a_n}^{kb_n}(p^{a_n}):\ 0\leq k\leq a_n-1\}$; assume that $0\in(p^{a_n},g_{a_n}^{lb_n}(p^{a_n}))=(p^{a_n},q^{a_n})$. Let $G:[g_{a_n}(p^{a_n}),g_{a_n}(q^{a_n})]\rightarrow [p^{a_n},q^{a_n}]$, $G=g_{a_n}^{b^{a_n}-1}$. Again we know that from the iterates $g_{a_n}^{ib_n}([p^{a_n},q^{a_n}])$, $0\leq i\leq a_n-1$, only finitely many (independent of $a_n$) are outside $\cup_{i=0}^{b_n-1}g_n^iU$, while the other are bounded away from $0$ by $\delta_0$ again independently of $a_n$, and using Lemma \ref{intervals} we can conclude that $G$ is bi-Lipschitz with some Lipschitz constant $L$ which is independent of $a_n$.

We have to prove that $G(c)$ and $G(d)$ are bounded away from $0$ independently from $a_n$. We will show that if one of them is too close to zero, then the positive divergence condition at the saddle and the Lipschitz bounds for $G$ would force the existence of another periodic orbit for $g_{a_n}$, which would be a contradiction. Let $\delta>0$ such that $2\delta<\frac 1L\delta^{\frac 1r}$. Assume that $0<G(c)<\delta$, the case $0>G(d)>-\delta$ is similar. Then $g_{a_n}(-\delta,0)=(c-\delta^{\frac 1r},c)$, and $g_{a_n}^{b^{a_n}}((-\delta,o))=(G(c-\delta^{\frac 1r}),G(c))$. But $G(c)-G(c-\delta^{\frac 1r})\geq \frac 1L\delta^{\frac 1r}$, and because $G(c)<\delta$ we get $G(c-\delta^{\frac 1r})<\delta-\frac 1L\delta^{\frac 1r}<-\delta$. But this shows that $[-\delta,0]\subset g_{a_n}^{b^{a_n}}((-\delta,0))$, and this is a contradiction because it would imply the existence of another periodic point for $g_{a_n}$.
\end{proof}

An interesting question is how the (physical) invariant measures depend on the flow, by considering families of flows on the torus with a saddle and a source. Clearly one could not expect continuity at flows corresponding to rational rotation numbers, because we may have an attracting and a repelling periodic orbit annihilating each other. But are the Cherry flows points of continuity?

There is some work on generalized Cherry flows, i.e. flows with several saddles and sinks or sources, possibly on higher genus surfaces, see for example \cite{MS}, \cite{MR}. It is possible that some of the results in this paper can be extended in some of these more general settings.

{\bf Acknowledgments:} The first author was partially supported by the Marie Curie grant IEF-234559, and would like to thank CRM Barcelona and IME-USP Sao Paulo for their hospitality, and Jiagang Yang for useful conversations. The second author was supported by the CNPq grant 301534/2008-0.

\bibliographystyle{plain}

\end{document}